\documentclass[a4paper,11pt,leqno]{amsart}
\calclayout
\usepackage[matrix,arrow,tips,curve]{xy}
\usepackage{pictexwd, dcpic} 
\usepackage[english]{babel}
\usepackage{amsmath}
\usepackage{amssymb}    
\usepackage{graphicx}
\usepackage{mathrsfs}
\usepackage{enumerate} 
\usepackage{psfrag}

\usepackage{hyperref}

 \psfrag{E}{\fontsize{20}{20}$E$}
 \psfrag{F}{\fontsize{20}{20}$F$}
 \psfrag{B}{\fontsize{20}{20}$\Delta_0$}
 \psfrag{l}{\fontsize{20}{20}$\ell$}
 \psfrag{E5}{\fontsize{20}{20}$\wi{E}$}
 \psfrag{E0}{\fontsize{20}{20}$E_0$}
 \psfrag{E1}{\fontsize{20}{20}$E_1$}
 \psfrag{E2}{\fontsize{20}{20}$E_2$}
 \psfrag{A1}{\fontsize{20}{20}$A_1$}
 \psfrag{A2}{\fontsize{20}{20}$A_2$}
 \psfrag{G1}{\fontsize{20}{20}$E_0'$}
 \psfrag{G}{\fontsize{20}{20}$E_0'$}
 \psfrag{G2}{\fontsize{20}{20}$E_0''$}
 \psfrag{D1}{\fontsize{20}{20}$D_1$}
 \psfrag{D}{\fontsize{20}{20}$D$}
 \psfrag{D2}{\fontsize{20}{20}$D_2$}
 \psfrag{R1}{\fontsize{20}{20}$R_1$}
 \psfrag{R2}{\fontsize{20}{20}$\w{R}_2$}
 \psfrag{R3}{\fontsize{20}{20}$\w{R}_3$}
 \psfrag{X3}{\fontsize{20}{20}$X_3$}
 \psfrag{R4}{\fontsize{20}{20}$\w{R}_4$}
 \psfrag{E3}{\fontsize{20}{20}$E_3$}
 \psfrag{E4}{\fontsize{20}{20}$E_4$}
 \psfrag{S3}{\fontsize{20}{20}$\sigma(E_3)$}
 \psfrag{S4}{\fontsize{20}{20}$\sigma(E_4)$}

\newtheorem{thm}{Theorem}[section]
\newtheorem{lemma}[thm]{Lemma}
\newtheorem{proposition}[thm]{Proposition}

\newtheorem{construction}{Construction}

\theoremstyle{definition}
\newtheorem{remark}[thm]{Remark}

\newcommand{\w}{\widetilde}

\newcommand{\wi}{\widehat}

\title[A note on flatness of some fiber type contractions]{A note on flatness of some fiber type contractions}
\author{E. A.~Romano}
\address{Istitute of Mathematics,
Faculty of Informatics, Mathematics and Mechanics,
University of Warsaw, ul. Banacha 2,
PL-02-097 Warszawa, Poland.}
\email{elrom@mimuw.edu.pl}
\begin{document}
\thanks{The project has been supported by Polish National Science
  Center grants 2013/08/A/ST1/00804 and 2016/23/G/ST1/04828. Thanks to Cinzia Casagrande for important discussions, and to the referees for their comments, expecially concerning Remark \ref{rem}.}
\maketitle
\begin{abstract} We discuss the flatness property of some fiber type contractions of complex smooth projective varieties of arbitrary dimensions. We relate the flatness of some morphisms having one-dimensional fibers with their conic bundles structures, also in the general case in which some mild singularities of the varieties are admitted.
\end{abstract}
\section{Introduction}
Let $X$ be a complex normal projective variety of arbitrary dimension $n$. A \emph{contraction} of $X$ is a surjective morpism $\varphi\colon X\to Y$ with connected fibers, where $Y$ is a  complex normal projective variety. We say that $\varphi$ is $K_X$-negative (or simply $K$-negative) if $K_X$ is $\mathbb{Q}$-Cartier and it has negative intersection with each curve contracted by $\varphi$. 

In this note we deal with the case in which such a contraction $\varphi$ is of fiber type, namely $\dim{X}>\dim{Y}$.

A \textit{conic bundle} $f\colon X\to Y$ is a $K_X$-negative fiber type contraction where $X$ is smooth and whose fibers are isomorphic to plane conics; \textit{i.e.} every fiber is isomorphic as a scheme to a zero locus of a non-trivial section of $\mathcal{O}_{\mathbb{P}^{2}}(2)$. 
We refer the reader to \cite{IO,IO2,M_R} for a recent account on conic bundles, see also references therein. By \cite[IV.15.4.2.]{GROT}) a conic bundle $f\colon X\to Y$ is a flat morphism, \textit{i.e.} for every $x\in X$ the stalk $\mathcal{O}_{X,x}$ is a flat $\mathcal{O}_{Y,f(x)}$-module.  


In this note we show the flatness property of some $K$-negative fiber type contractions. On the other hand, using the flatness of the morphisms in question, we prove that they have a conic bundle structure. The starting point is the following theorem which is due to Ando (see \cite[Theorem 3.1 (ii)]{ANDO}) and it is a generalization in higher dimension of Mori's result in dimension $3$ (see \cite[Theorem 3.5, (3.5.1)]{MORI}). 

\begin{thm} \label{Ando} Let $X$ be a complex smooth projective variety and let $f\colon X\to Y$ be a $K$-negative contraction where every fiber is one-dimensional. Then $Y$ is also smooth and $f$ is a conic bundle.
	
\end{thm}
	
The goal of this paper is to discuss an alternative proof of the above theorem by proving the flatness of the contraction. Indeed, this property is not analyzed in \cite[Theorem 3.1 (ii)]{ANDO} but it represents a key point to deduce the smoothness of $Y$ and the conic bundle structure of $f$. See also our motivation explained in Remark \ref{remark_Ando}. In particular, we are going to prove the following result, from which we discuss in Remark \ref{remark_Ando} how one can deduce Theorem \ref{Ando}. 
\begin{thm} \label{thm_flatness} Let $X$ be a complex smooth projective variety, and let $f\colon X\to Y$ be a $K$-negative contraction where every fiber is one dimensional. Then $f$ is a flat morphism.
\end{thm}

Finally, we focus on the more general case in which a complex normal projective variety $X$ with mild singularities admits a flat fiber type $K$-negative contraction with one-dimensional fibers. Let us recall that given such a variety $X$, it is \textit{Gorenstein} if it is Cohen-Macaulay and $K_X$ is a Cartier divisor. Being Cohen-Macaulay is an algebraic condition on the local rings of $X$; we refer the reader to \cite[II, $\S$8]{HART} and \cite[VII, $\S$17]{MAT} for its definition and properties.  
In the singular case, our main statement is as follows.  
\begin{proposition} \label{conic_singular}
Let $f\colon X\to Y$ be a fiber type $K_X$-negative contraction, where every fiber of $f$ is one-dimensional, $X$ is a Gorenstein projective variety with log-terminal singularities, and $Y$ is smooth. Then $f_{*}\mathcal{O}_{X}(-K_{X})$ is a locally free sheaf on $Y$ of rank $3$, $X$ can be embedded into $\mathbb{P}(f_{*}\mathcal{O}_{X}(-K_{X}))$, and through this immersion the fibers of $f$ are isomorphic to plane conics.
\end{proposition}

\section{Flatness and conic bundles structures}

The first part of this section is devoted to prove Theorem \ref{thm_flatness}. To this end, the main idea is to consider a birational modification of $X$ to make flat the morphism. We are going to generalize to higher dimension the strategy used by Mori to prove \cite[Lemma 3.25, Lemma 3.26]{MORI}. We recall that $X$ has arbitrary dimension $n$. The following construction represents the set-up for the proof of Theorem \ref{thm_flatness}. 

\begin{construction} [\cite{MORI}, (3.24)] \label{construction} 
	{\normalfont In the setting of Theorem \ref{thm_flatness}, let us consider $Y_{0}\subset Y$ the maximal open set such that $f_{\mid f^{-1}(Y_{0})}\colon f^{-1}(Y_{0})\to Y_{0}$ is a flat morphism. 
		
	Set $X_{0}:=f^{-1}(Y_{0})$. Denoting by $\text{Hilb}(X)$ the Hilbert scheme of $X$, since $f_{\mid f^{-1}(Y_{0})}$ is flat, we have an injective map $Y_{0}\to \text{Hilb}(X)$. Up to restricting $Y_{0}$, this map gives an isomorphism to an open subset of $\text{Hilb}(X)$. For simplicity of notation, we continue to denote this restriction by $Y_{0}$. 
	
	Let $\widetilde{Y}$ be the closure of $Y_{0}$ in $\text{Hilb}(X)$ with the reduced subscheme structure. 
	
	Let us consider the universal family $Z\subset X\times \widetilde{Y}$, and the two natural projections $\pi\colon Z\to X$ and $\phi\colon Z\to \widetilde{Y}$, where $\pi$ is birational, being an isomorphism over $X_{0}$, and $\phi$ is a flat morphism by construction. We have the following diagram:
	
	$$
	\xymatrix{
		&{Z}\subset X\times \widetilde{Y}\ar[dl]_\phi \ar[dr]^\pi&\\
		\widetilde{Y}&&{X}
	}$$
		
	Moreover, by Lemma \cite[Theorem 3.1 (i)]{ANDO} the general fiber of $\phi$ is isomorphic to $\mathbb{P}^{1}$. Notice that $Z$ is an irreducible and reduced projective variety, because $\phi$ is flat, $\widetilde{Y}$ is irreducible and reduced, and the general fiber of $\phi$ is isomorphic to $\mathbb{P}^{1}$. 
		
	For every $y\in \widetilde{Y}$, set $Z_{y}:=\phi^{-1}(y)\subset X\times\{y\}$. We know that $Z_{y}$ is a subscheme of $X$, and as a $1$-cycle is algebraically equivalent to the general fiber of $f$, so that for every $y\in \widetilde{Y}$, $Z_{y}$ is contracted to a point by $f$. Hence, for every $y\in \widetilde{Y}$, $Z_{y}\subset \widetilde{F}$, where $\widetilde{F}$ is a fiber of $f$. By our assumption $\dim{\widetilde{F}}=1$, and $-K_{X}\cdot Z_{y}=-K_{X}\cdot \widetilde{F}=2$, so that $Z_{y}=\widetilde{F}$ as $1$-cycles, and $(Z_{y})_{red}= \widetilde{F}_{red}$.} 
\end{construction}

   \vspace{0,5cm}
    In order to prove Theorem \ref{thm_flatness} we need to study the fibers of $\phi$. To make the exposition self contained we recall the following two lemmas. In the first one all the possibilities for such fibers are listed. The second lemma is a short version of \cite[Lemma 1.5]{ANDO}. We refer the reader also to \cite[Lemma 3.1.7, Lemma 3.1.8]{TESI} for detailed proofs of the following results.  
   \begin{lemma} [\cite{MORI}, Lemma (3.25)]\label{fibers_phi}
   	Setting as in Construction \ref{construction}. Take $y\in \widetilde{Y}$. Then $Z_{y}$ does not have embedded points, and there are three possibilities for $Z_{y}$:
   	\begin{enumerate}
   		\item [(a)] $Z_{y}\cong \mathbb{P}^{1}$;
   		\item [(b)] $Z_{y}\cong C_{1}\cup C_{2}$, where $C_{i}$ are distinct components, both isomorphic to $\mathbb{P}^{1}$, and they intersect transversally at a point;
   		\item [(c)] $Z_{y}=2C$ as $1$-cycles, with $C\cong\mathbb{P}^{1}$.
   	\end{enumerate}
   	 
   	In case $(c)$, one has that $I_{C}^{2}\subset I_{Z_{y}}\subset I_{C}$, where  $I_{Z_{y}}$ denotes the ideal subsheaf of $\mathcal{O}_{X}$ defining $Z_{y}$ in $X$, and similarly for $I_{C}$.
   \end{lemma} 
   \begin{remark} \label{rem} Assume that case $(c)$ of Lemma \ref{fibers_phi} holds, then $I_C/I_{Z_y}\cong \mathcal{O}_{C}(-1)$. Indeed, $I_C/I_{Z_y}$ is an $\mathcal{O}_{C}$-module of rank 1 without embedded points, that is, an $\mathcal{O}_{C}$-invertible sheaf. Since $\phi\colon Z\to \tilde{Y}$ is flat, $\chi(\mathcal{O}_{Z_y})$ is independent of $y\in \tilde{Y}$ and $\chi(\mathcal{O}_{Z_y})=1$ because $Z_{\eta}\cong \mathbb{P}^{1}$ for a general geometric point $\eta\in \tilde{Y}$. Hence $I_C/I_{Z_y}\cong \mathcal{O}_{C}(-1)$ as claimed by $\chi(I_C/I_{Z_y})=\chi(\mathcal{O}_{Z_y})-\chi(\mathcal{O}_C)=0$.
   \end{remark} 
    
    \begin{lemma} \label{claim_2} Setting as in Construction \ref{construction}. Assume that $Z_{y}=2C$ as $1$-cycles.
    	There are two possibilities for the conormal sheaf of $C$ in $X$:
    	\begin{enumerate}
    		\item [(a)] $I_{C}/I_{C}^2\cong \mathcal{O}_{C}(-1)\oplus\mathcal{O}_{C}^{n-3}\oplus  \mathcal{O}_{C}(2)$;
    		\item [(b)] $I_{C}/I_{C}^2\cong \mathcal{O}_{C}(-1)\oplus\mathcal{O}_{C}^{n-4}\oplus  \mathcal{O}_{C}(1)^{\oplus 2}$.
    	\end{enumerate}
    	 If $(a)$ holds, then $I_{Z_{y}}/I_{C^{2}}\cong \mathcal{O}_{C}^{n-3}\oplus  \mathcal{O}_{C}(2)$. Otherwise, $I_{Z_{y}}/I_{C^{2}}\cong \mathcal{O}_{C}^{n-4}\oplus  \mathcal{O}_{C}(1)^{\oplus 2}$.
    \end{lemma}

\begin{proof}[Proof of Theorem \ref{thm_flatness}] Let us consider the set-up as in Construction \ref{construction}. We prove that $\pi$ is an isomorphism. Since $\pi$ is birational and $X$ is normal, it is enough to show that $\pi$ is finite. Assume by contradiction that there exists an irreducible curve $\Gamma\subset Z$ such that $\pi(\Gamma)=\{x_{0}\}$ with $x_{0}$ point of $X$. 

We see that $\phi(\Gamma):=\tilde{\Gamma}$ gives a one-dimensional family of subschemes of $X$. Let us consider $S:=\phi^{-1}(\tilde{\Gamma})$ such a family, which contains $\Gamma$. Then $\dim{\pi(S)}=1$. 

Indeed, since $\pi(\Gamma)=\{x_{0}\}$, the $1$-cycles $Z_{y}$ for $y\in \tilde{\Gamma}$ pass through $x_{0}$, and they are contracted by $f$. Thus there exists a fiber $F_{0}\subset X$ of $f$ such that $Z_{y}\subseteq F_{0}$, and  $(Z_{y})_{red}=(F_{0})_{red}$ for every $y\in \tilde{\Gamma}$. Denote by $\tilde{y}$ the general point of $\tilde{\Gamma}$. 

If $Z_{\tilde{y}}$ is reduced, then $Z_{\tilde{y}}=F_{0}$ which is a contradiction, because $\tilde{\Gamma}\subset \text{Hilb}{(X)}$ so that for $y\in\tilde{\Gamma}$ the subschemes $Z_{y}$ are distinct.

Then $Z_{y}$ is not reduced for every $y\in \tilde{\Gamma}$, and $(Z_{y})_{red}=(F_{0})_{red}=C$ with $C\cong \mathbb{P}^{1}$. 
   By Lemma \ref{fibers_phi}, we know that $I_{C}^{2}\subset I_{Z_{y}}\subset I_{C}$. We use the following exact sequence:
  \begin{equation} \label{standard}
  0\longrightarrow I_{Z_{y}}/I_{C}^{2}\longrightarrow I_{C}/I_{C}^{2}\longrightarrow I_{C}/I_{Z_{y}}\longrightarrow 0
  \end{equation}
  where $I_{Z_{y}}$ denotes the ideal subsheaf of $\mathcal{O}_{X}$ defining $Z_{y}$ in $X$, and similarly for $I_{C}$.  Assume that we are in case $(a)$ of Lemma \ref{claim_2}, so that one has $I_{C}/I_{C}^{2}\cong \mathcal{O}_{C}(-1)\oplus\mathcal{O}_{C}^{n-3}\oplus  \mathcal{O}_{C}(2)$, and by Remark \ref{rem} we know that $I_{C}/I_{Z_{y}}\cong \mathcal{O}_{C}(-1)$. Then from (\ref{standard}) we get the following exact sequence
  \begin{equation} \label{split}
  	0\longrightarrow I_{Z_{y}}/I_{C}^{2}\stackrel{\alpha}{\longrightarrow} \mathcal{O}_{C}(-1)\oplus\mathcal{O}_{C}^{n-3}\oplus  \mathcal{O}_{C}(2)\stackrel{\beta}{\longrightarrow} \mathcal{O}_{C}(-1)\longrightarrow 0
  \end{equation}
  
  By the isomorphisms $\text{Hom}(\mathcal{O}_{C}(-1)\oplus\mathcal{O}_{C}^{n-3}\oplus  \mathcal{O}_{C}(2), \mathcal{O}_{C}(-1))\cong H^{0}(\mathcal{O}\oplus \mathcal{O}(-1)^{\oplus n-3}\oplus \mathcal{O}(-3))\cong \mathbb{C}$, it follows that up to multiplication by scalars the map $\beta$ of (\ref{split}) is unique. 

Thus the subsheaf $I_{Z_{y}}/I_{C}^{2}\subset I_{C}/I_{C}^{2}$ is uniquely determined, then also $I_{Z_{y}}\subset I_{C}$ is uniquely determined. This contradicts the non constancy of the family $S\to \tilde{\Gamma}$. Repeating the same argument, we get a contradiction also when Lemma \ref{claim_2} $(b)$ holds.

Then $\pi\colon Z\to X$ is an isomorphism. Since $\phi$ is a flat morphism and $X$ is smooth, by \cite[Proposition 17.3.3 (i)]{GROT} it follows that $\widetilde{Y}$ is also smooth. 

Applying Lemma \ref{fibers_phi}, we deduce that $\phi$ and $f$ have the same fibers, so that using \cite[Proposition 1.14]{DEB}, we find that $\phi=f$, and $\widetilde{Y}\cong Y$. Then $Y$ is smooth and $f$ is a flat morphism. 
\end{proof}

\begin{remark} \label{remark_Ando}
	In \cite[Theorem 3.1 (ii)]{ANDO} Ando proved Theorem \ref{Ando}. He showed that the fibers of $f$ have at most two irreducible components, and he analized the three possible cases for the fibers to prove that they are isomorphic to plane conics. 
	
	To this end, when the fiber $F=2C$ as $1$-cycle with $C\cong \mathbb{P}^{1}$, in \cite[pag. 356, case 3]{ANDO} Ando claims that $\chi(\mathcal{O}_{F})=1$. We could not understand how to deduce that $\chi(\mathcal{O}_{F})=1$, without knowing that $f$ is flat. For this reason, to get Theorem \ref{Ando}, first we need to prove Theorem \ref{thm_flatness}. 
	Then the proof of Theorem \ref{Ando} runs as done in \cite[Theorem 3.1 (ii)]{ANDO}. 
\end{remark}
\vspace{0,2cm}
Now we show Proposition \ref{conic_singular}. This is probably well-known to experts, but we include a proof for lack of references. To this end, we start with the following easy observation. 

   \begin{remark} \label{flat2} Let $f\colon X\to Y$ be a flat morphism between projective varieties and let $\mathcal{F}$ be a coherent locally free sheaf on $X$. Then $\mathcal{F}$ is flat over $Y$, namely for every $x\in X$, the stalk $\mathcal{F}_{x}$ is a flat $\mathcal{O}_{Y,f(x)}$-module\footnote{Notice that we consider $\mathcal{F}_{x}$ as an $\mathcal{O}_{Y,f(x)}$-module via the natural map $\mathcal{O}_{Y,f(x)}\to \mathcal{O}_{X,x}$.}.
\end{remark}
\begin{proof}  By \cite[III, Proposition 9.2 (e)]{HART} it follows that for every $x\in X$, $\mathcal{F}_{x}$ is a flat $\mathcal{O}_{X,x}$-module. Since $f$ is flat, by definition $\mathcal{O}_{X,x}$ is a flat $\mathcal{O}_{Y,f(x)}$- module, hence by the property of transitivity of flat sheaves (see \cite[III, Proposition 9.2 (c)]{HART}), it follows that for every $x\in X$, $\mathcal{F}_{x}$ is a flat $\mathcal{O}_{Y,f(x)}$-module, hence the statement. 
\end{proof}
\begin{proof}[Proof of Proposition \ref{conic_singular}]
We notice that $f$ is an equidimensional morphism from a Cohen-Macaulay variety to a smooth variety, so that it is flat (see for instance \cite[Corollary of Theorem 23.1]{MAT}). Using the same proof of \cite[Lemma 3.1 (i)]{ANDO}, one can see that the general fiber of $f$ is isomorphic to $\mathbb{P}^{1}$. We show that the fibers are isomorphic to plane conics. Set $\mathcal{E}:=f_{*}\mathcal{O}_{X}(-K_{X})$. We prove that $\mathcal{E}$ is a locally free sheaf on $Y$ of rank 3. Set $\mathcal{F}:= \mathcal{O}_{X}(-K_{X})$, that by Remark \ref{flat2} is flat over $Y$. If we denote by $X_{y}$ the fiber over $y \in Y$, since $f$ is a flat morphism, we know that $\chi{(X_{y},\mathcal{F}_{y})}=h^{0}(X_{y}, \mathcal{F}_{y})-h^{1}(X_{y}, \mathcal{F}_{y})$ is constant. Using the same argument of \cite[(3.25.1)]{MORI} it is easy to check that $h^{1}(X_{y}, \mathcal{F}_{y})=0$ for every $y\in Y$. 


Then $\chi{(X_{y},\mathcal{F}_{y})}=h^{0}(X_{y},\mathcal{F}_{y})=h^{0}(\mathbb{P}^{1}, \mathcal{O}_{\mathbb{P}^{2}}(2))=3$. 

Hence by \cite[Proposition 3.1.9]{HART} we deduce that $\mathcal{E}$ is locally free sheaf on $Y$ of rank 3, so that we have a $\mathbb{P}^{2}$-bundle $\pi\colon \mathbb{P}(\mathcal{E})\to Y$. Since every fiber is one-dimensional, using \cite[Corollary 1.11.1]{ANDREATTA}, it follows that $\mathcal{F}$ is $f$-very ample on $X$ so that $f^{*}f_{*}\mathcal{F}$ gives an immersion $X\hookrightarrow \mathbb{P}(\mathcal{E})$ over $Y$, and $f=\pi_{\mid X}$. Being $\dim{\mathbb{P}(\mathcal{E})}=\dim{Y}+2=\dim{X}+1$, $X$ is embedded as a divisor of $\mathbb{P}(\mathcal{E})$. 

Finally, we prove that the restriction of $X$ to every fiber of $\pi$ is one-dimensional and belongs to $|\mathcal{O}_{\mathbb{P}^{2}}(2)|$. Denoting by $l$ a fiber of $f$, we have that $l$ is the intersection between $X$ and a fiber $l^{\prime}$ of $\pi$, hence $l^{\prime}\cong \mathbb{P}^{2}$. Now, using that $X$ is a divisor and that $-K_{X}$ has degree 2 on every fiber of $f$, we get $l= X_{{\mid}l^{\prime}}=\mathcal{O}_{\mathbb{P}^{2}}(2)$, hence our claim. 
\end{proof}
%
\bibliographystyle{plain} 
\bibliography{main.bib}
\end{document}